\documentclass[smallextended,envcountsect]{svjour3}
\usepackage{amsmath,amsfonts,amssymb}
\smartqed
\usepackage{ifthen}
\usepackage{longtable}
\usepackage{array}
\usepackage{url}
\usepackage{multirow}
\usepackage{graphicx}
\usepackage{booktabs}
\usepackage{float}

\begin{document}
\title{Optimal Bounds for Integrals with Respect to Copulas and Applications\thanks{Communicated by Paul I. Barton}}
\date{}
\author{Markus Hofer \and Maria Rita Iac\`o\thanks {The second author is funded by the fellowship of the Doctoral School in Mathematics and Computer Science of University of Calabria and is partially supported by the Austrian Science Fund (FWF): W1230, Doctoral Program ``Discrete Mathematics''}}
\institute{Markus Hofer (corresponding author)
\at Graz University of Technology, Institute of Mathematics A, Steyrergasse 30, 8010 Graz, Austria\\ \email{markus.hofer@tugraz.at}
\and
Maria Rita Iac\`o
\at Graz University of Technology,
Institute of Mathematics A, Steyrergasse 30, 8010 Graz, Austria
\at University of Calabria, Department of Mathematics and Computer Science, Via P. Bucci 30B, 87036 Arcavacata di Rende (CS), Italy\\
\email{iaco@math.tugraz.at} 
}
\maketitle

\begin{abstract}
 We consider the integration of two-dimensional, piecewise constant functions with respect to copulas. By drawing a connection to linear assignment problems, we can give optimal upper and lower bounds for such integrals and construct the copulas for which these bounds are attained. Furthermore, we show how our approach can be extended in order to approximate extremal values in very general situations. Finally, we apply our approximation technique to problems in financial mathematics and uniform distribution theory, such as the model-independent pricing of first-to-default swaps.
\end{abstract}

\section{Introduction}\label{intro}
In the last decades, the importance of copulas in mathematical modeling was recognized by many researchers; see e.g.\ \cite{tank,pr1,packham}. Many applications come from actuarial and financial mathematics, where the joint distribution of a vector of random variables is studied frequently. Typical problems are the pricing of basket options or the derivation of the Value at Risk of a portfolio. In this context, an interesting question concerns the best or worst case when the marginal distributions are given but the dependence structure of the underlying random vector is unknown or only partially known. Such situations appear frequently since dependence structures are in general more difficult to calibrate from empirical data than marginal distributions. Thus we are interested in maximizing the value of an integral by considering all possible copulas as integrators.\\

The underlying problem is in general open, however there exist solutions for some particular classes of integrand functions $f$. For instance, Rapuch and Roncalli \cite{rap} consider basket option pricing when no information on the dependence of the underlying random variables is available. They derive bounds for the prices of several options of European type in the Black-Scholes model, where the integrand function has a mixed second derivative with constant sign on the unit square. Tankov \cite{tank} extends these results to the greater class of two-increasing (or supermodular) functions $f$, a definition will be given in the next section. Furthermore the author gives an extension to option pricing problems under partial information on the dependence of the underlying random variables. Note that the above results are based on classical findings due to Tchen \cite{tchen}.\\

Similar results and applications in number theory are presented by Fialov\'a and Strauch \cite{fial}. They consider bounds for functionals which depend on two uniformly distributed point sequences. Under similar conditions as in \cite{rap}, they show that the Fr\'echet-Hoeffding bounds $W$ and $M$ are the copulas for which the extremal values are obtained. We remark that the underlying problem was formulated as an open problem in the unsolved problem collection of \emph{Uniform Distribution Theory} \footnote{Problem 1.29 in the open problem collection as of 19. January 2013 (http://www.boku.ac.at/MATH/udt/unsolvedproblems.pdf)}. A more detailed introduction to applications in uniform distribution theory is given in Section \ref{udt} of this article.\\ 

A list of results for a different class of functions $f$ exists in the context of financial risk theory; see e.g.\ Puccetti and R\"uschendorf \cite{pr1} or Albrecher et al.\ \cite{aak}. In \cite{pr1} the authors derive sharp bounds for quantiles of the loss of a portfolio, represented by a finite sum of dependent random variables, when no or only partial information on the dependence structure within the portfolio is available. Such quantities play an important role in actuarial and financial mathematics, for instance in the computation of the Value at Risk. Recently this approach has been generalized to derive optimal bounds for the expected shortfall of a portfolio; see Puccetti \cite{pr3}. Many of these results rely on the so-called rearrangement method due to R\"uschendorf \cite{ru}. Note that the application of the rearrangement method requires a rather strong regularity of the integrand function; see e.g.\ \cite{pr3}. The optimal bounds in the articles mentioned above are attained by using so-called shuffles of $M$-class of copulas, which we define in the next section.\\

The structure of our paper is the following: in the next section, after a short introduction to copulas, we present our main results, which are bounds on integrals of piecewise constant functions. Furthermore, we formulate an approximation technique for a very general class of integrand functions. In the third section, we apply our results to problems in uniform distribution theory and financial mathematics.

\section{Main Results}\label{mainsec}

In the sequel we consider expectations,
\begin{equation}\label{start}
 \mathbb{E}[f(X,Y)],
\end{equation}
where $f$ is a function on $[0,1[^2$ and $X,Y$ are uniformly distributed random variables on the unit interval. In this situation the joint distribution function $C$ of $X$ and $Y$ is a copula. 

\begin{definition}[Copula]\label{cop}
 Let $C$ be a positive function on the unit square. Then $C$ is called (two)-copula iff for every $x,y \in [0,1[$
\begin{align*}
 C(x,0) &= C(0,y) = 0,\\
 C(x,1) &= x \text{ and } C(1,y) = y,
\end{align*}
and for every $x_1,x_2,y_1,y_2 \in [0,1[$ with $x_2 \geq x_1$ and $y_2 \geq y_1$
\begin{equation}\label{two-inc}
 C(x_2,y_2) - C(x_2, y_1) - C(x_1, y_2) + C(x_1,y_1) \geq 0.
\end{equation}
A function which satisfies \eqref{two-inc} is called two-increasing or supermodular. In the sequel we denote by $\mathcal{C}$ the set of all two-copulas.
\end{definition}

Note that the restriction to uniformly distributed marginals is insignificant since by Sklar's Theorem, see e.g.\ \cite[Theorem 2.3.3]{nelsen}, we can write every continuous two-dimensional distribution function $H$ as
\begin{equation*}
 H(x,y) = C(F(x),G(y)),
\end{equation*}
where $F,G$ denote the marginal distributions of $H$ and $C$ is a copula. Moreover if $F$ and $G$ are continuous, then $C$ is unique and we have
\begin{equation*}
\int_{[0,1[^2} f(x,y) dH(x,y) =  \int_{[0,1[^2} f(F^{-1}(x),G^{-1}(y)) dC(x,y),
\end{equation*}
where $F^{-1},G^{-1}$ denote the inverse distribution functions of the marginals.\\

Copulas can be ordered stochastically, where the upper and lower bounds are called Fr\'echet-Hoeffding bounds (see e.g.\ \cite[Theorem 2.2.3]{nelsen}). More precisely, for every two-copula $C$ we have
\begin{equation}\label{frechet}
 \max(x+y-1,0) \leq C(x,y) \leq \min(x,y), \quad \text{ for all } (x,y) \in [0,1[^2.
\end{equation}
It is also well-known that the Fr\'echet-Hoeffding lower and upper bounds $W(x,y) = \max(x+y-1,0)$ and $M(x,y) = \min(x,y)$ are copulas in the two dimensional setting. For higher dimensions an analogon of \eqref{frechet} exists, however the lower bound is in general not a copula, see \cite[Theorem 3.2 and 3.3]{joe}. For a detailed introduction to copulas see \cite{nelsen,joe}.\\ 

Thus, according to the discussion in the beginning of Section \ref{intro}, we are interested in bounds of the form
\begin{equation}\label{bounds}
  \int_{[0,1[^2} f(x,y) dC_{\min}(x,y) \leq  \int_{[0,1[^2} f(x,y) dC(x,y) \leq  \int_{[0,1[^2} f(x,y) dC_{\max}(x,y),
\end{equation}
for all $C \in \mathcal{C}$, where $C_{\min}, C_{\max}$ are copulas. As mentioned above a particularly interesting subclass of copulas for our problems are so-called shuffles of $M$, see \cite[Section 3.2.3]{nelsen}.\\

\begin{definition}[Shuffles of $M$]\label{shuf}
Let $n \geq 1$, $s = (s_0, \ldots, s_n)$ be a partition of the unit interval with $0 = s_0 < s_1 < \ldots < s_n = 1$, $\pi$ be a permutation of $S_n = \{1,\ldots,n\}$ and $\omega \colon S_n \rightarrow \{-1, 1\}$. We define the partition $t= (t_0, \ldots, t_n), ~0 = t_0 < t_1 < \ldots < t_n = 1$ such that each $[s_{i-1}, s_i[ \times [t_{\pi(i)-1}, t_{\pi(i)}[$ is a square. A copula $C$ is called shuffle of $M$ with parameters $\{n, s, \pi, \omega\}$ if it is defined in the following way: for all $i \in \{1, \ldots, n\}$ if $\omega(i) = 1$, then $C$ distributes a mass of $s_i - s_{i-1}$ uniformly spread along the diagonal of $[s_{i-1}, s_i[ \times [t_{\pi(i)-1}, t_{\pi(i)}[$ and if $\omega(i) = -1$ then $C$ distributes a mass of $s_i - s_{i-1}$ uniformly spread along the antidiagonal of $[s_{i-1}, s_i[ \times [t_{\pi(i)-1}, t_{\pi(i)}[$.
\end{definition}

Note that the two Fr\'echet-Hoeffding bounds $W, M$ are trivial shuffles of $M$ with parameters $\{1, (0,1), (1), -1\}$ and $\{1, (0,1), (1), 1\}$, respectively. Furthermore, it is well-known that every copula can be approximated arbitrarily close with respect to the supremum norm by a shuffle of $M$; see e.g.\ \cite[Theorem 3.2.2]{nelsen}. In the sequel we denote by $I_n$ the partition of the unit interval which consists of $n$ intervals of equal length.\\

In next theorem we illustrate the close relation of \eqref{bounds} to problems in optimization theory, namely linear assignment problems of the form
\begin{equation}\label{lin}
 \max_{\pi \in \mathcal{P}} \sum_{i = 1}^n a_{i, \pi(i)},
\end{equation}
where $\mathcal{P}$ is the set of all permutations of $\{1, \ldots, n\}$. Such problems are well understood and can be solved efficiently, for example by using the celebrated Hungarian Algorithm due to Kuhn \cite{kuhn}. For a detailed description of assignment problems and related solution algorithms we refer to \cite{burk}.

\begin{theorem}\label{main1}
Let $n \geq 1$, $A = \{a_{i,j}\}_{i,j=1, \ldots,n}$ be a real-valued $n \times n$ matrix and let the function $f$ be defined as
\begin{equation*}
 f(x,y) := a_{i,j}, \quad (x,y) \in \left[ \frac{i-1}{n}, \frac{i}{n} \right [ \times \left[ \frac{j-1}{n}, \frac{j}{n} \right[.
\end{equation*}
 Then the copula $C_{\max}$ which maximizes
\begin{equation}\label{int}
 \max_{C \in \mathcal{C}} \int_{[0,1[^2} f(x,y) dC(x,y)
\end{equation}
 is given as a shuffle of $M$ with parameters $\{n, I_n, \pi^*, 1\}$, where $\pi^*$ is the permutation which solves the assignment problem
\begin{equation*}
 \max_{\pi \in \mathcal{P}} \sum_{i = 1}^n a_{i, \pi(i)}.
\end{equation*}
 Moreover, the maximal value of \eqref{int} is given as
\begin{equation}\label{maxval}
  \int_{[0,1[^2} f(x,y) dC_{\max}(x,y) = \frac{1}{n} \sum_{i = 1}^n a_{i, \pi^*(i)}.
\end{equation}
\end{theorem}

\begin{proof}
Let $\{C_k(x,y), k = 1,\ldots, n! = N\}$ be the set of all shuffles of $M$ with parameters of the form $\{n, I_n, \pi_k, 1\}$ and let $t_k \geq 0, ~k = 1, \ldots, N$, where $\sum_{k = 1}^N t_k = 1$. Then $C'(x,y) = \sum_{k = 1}^N t_k C_k(x,y)$ is always a copula satisfying 
\begin{equation*}
 \int_{[0,1[^2} f(x,y) dC'(x,y) \leq \frac{1}{n} \sum_{i = 1}^n a_{i, \pi^*(i)},
\end{equation*}
where $\pi^*$ is given in the statement of the theorem.\\

For an arbitrary copula $C \in \mathcal{C}$ we define the matrix $B_C$ as
\begin{equation*}
 B_C(i,j) = n \int_{\left [ \frac{i-1}{n}, \frac{i}{n} \right [ \times \left [ \frac{j-1}{n}, \frac{j}{n} \right [} dC(x,y).
\end{equation*}
It follows by Definition \ref{cop} that $B_C$ is doubly stochastic and by Definition \ref{shuf} that $B_{C_k}$ is a permutation matrix. Furthermore it follows by the Birkhoff-von Neumann Theorem that the set of doubly stochastic matrices coincides with the convex hull of the set of permutation matrices, see e.g. \cite{Mirsky}. Thus for every $C$ there exist $t_k \geq 0, ~k = 1,\ldots,N $ with $\sum_{k = 1}^N t_k = 1$ such that
\begin{equation*}
 B_C(i,j) = \sum_{k = 1}^N t_k B_{C_k}(i,j), \quad \text{ for every } i,j,
\end{equation*}
and hence
\begin{equation*}
 \int_{[0,1[^2} f(x,y) dC(x,y) = \sum_{k = 1}^N t_k \int_{[0,1[^2} f(x,y) dC_k(x,y) \leq \frac{1}{n} \sum_{i = 1}^n a_{i, \pi^*(i)}\ .\qquad \qed
\end{equation*}
\end{proof}

Note that the maximal copula in Theorem \ref{main1} is by no means unique, since for instance the value of the integral in \eqref{int} is independent of the choice of $\omega$.\\ 

Obviously, we can derive a lower bound in Theorem \ref{main1} by considering \linebreak$g(x,y) = -f(x,y)$. Furthermore, it is easy to see that Theorem \ref{main1} applies to all functions $f$ which are constant on sets of the form 
\begin{equation*}
 I_{i,j} = \left[ s_i, s_{i+1} \right[ \times \left[ t_j, t_{j+1} \right[, \quad i = 0, \ldots, n-1, ~ j = 0, \ldots, m-1,
\end{equation*}
where $0 = s_0 < s_1 < \ldots < s_n = 1$ and $0 = t_0 < t_1 < \ldots < t_m = 1$ are rational numbers.\\

The following generalization of our approach applies to a wide class of functions on the unit square.

\begin{theorem}\label{main2}
 Let $f$ be a continuous function on $[0,1]^2$, let the sets $I^n_{i,j}$ be given as
\begin{equation*} 
 I^n_{i,j} = \left[ \frac{i-1}{2^n}, \frac{i}{2^n} \right [ \times \left[ \frac{j-1}{2^n}, \frac{j}{2^n} \right[ \quad \text{ for } i,j = 1, \ldots, 2^n,
\end{equation*}
for every $n > 1$ and define the functions $\underline{f}_n, \overline{f}_n$ as
\begin{align}
 \underline{f}_n(x,y) &= \min_{(x,y) \in I^n_{i,j}} f\left( x, y \right), \quad \text{ for all } (x,y) \in I^n_{i,j},\notag\\
 \overline{f}_n(x,y) &= \max_{(x,y) \in I^n_{i,j}} f\left( x, y \right), \quad \text{ for all } (x,y) \in I^n_{i,j}.\label{mini}
\end{align}
 Furthermore, let $\underline{C}^n_{\max}, \overline{C}^n_{\max}$ be the copulas which maximize
\begin{equation*}
  \max_{C \in \mathcal{C}} \int_{[0,1[^2} \underline{f}_n(x,y) dC(x,y) \text{ and } \max_{C \in \mathcal{C}} \int_{[0,1[^2} \overline{f}_n(x,y) dC(x,y),
\end{equation*}
respectively. Then
\begin{align}
 \int_{[0,1[^2} \underline{f}_n(x,y) d\underline{C}^n_{\max}(x,y) &\leq \sup_{C \in \mathcal{C}} \int_{[0,1[^2} f(x,y) dC(x,y) \notag\\ 
  &\leq \int_{[0,1[^2} \overline{f}_n(x,y) d\overline{C}^n_{\max}(x,y),\label{bound}
\end{align}
for every $n$, and
\begin{align}
 \lim_{n \rightarrow \infty} \int_{[0,1[^2} \underline{f}_n(x,y) d\underline{C}^n_{\max}(x,y) &= \lim_{n \rightarrow \infty} \int_{[0,1[^2} \overline{f}_n(x,y) d\overline{C}^n_{\max}(x,y) \notag\\ 
&= \sup_{C \in \mathcal{C}} \int_{[0,1[^2} f(x,y) dC(x,y).\label{conv}
\end{align}
\end{theorem}
\begin{proof}
The inequalities in \eqref{bound} follow immediately from the construction of $\underline{f}_n,\overline{f}_n$ and Theorem \ref{main1}. Furthermore since $f$ is continuous on $[0,1]^2$ we have that for every $\epsilon > 0$ there exists an integer $n$ such that
\begin{equation}\label{eps}
 | \overline{f}_n(x,y) - \underline{f}_n(x,y) | < \epsilon, \quad \text{ for all } (x,y) \in [0,1]^2.
\end{equation}
Moreover, by Theorem \ref{main1}, for every $n$ we can write
\begin{equation*}
  \int_{[0,1[^2} \underline{f}_n(x,y) d\underline{C}^n(x,y) = \frac{1}{2^n} \sum_{i = 1}^{2^n} a_{i, \pi^*(i)},
\end{equation*}
for a permutation $\pi^*$ and a real valued matrix $A = \{a_{i,j}\}_{i,j=1, \ldots,n}$ with
\begin{equation*}
 a_{i,j} = \min_{(x,y) \in I^n_{i,j}} f\left( x, y \right), \quad \text{ for } i,j = 1, \ldots, 2^n.
\end{equation*}
Using \eqref{eps}, we get that
\begin{equation*}
  \int_{[0,1[^2} \overline{f}_n(x,y) d\overline{C}^n(x,y) \leq \int_{[0,1[^2} (\underline{f}_n(x,y) + \epsilon) d\underline{C}^n(x,y) = \frac{1}{2^n} \sum_{i = 1}^{2^n} (a_{i, \pi^*(i)} + \epsilon)
\end{equation*}
and thus
\begin{equation*}
  \left | \int_{[0,1[^2} \overline{f}_n(x,y) d\overline{C}^n(x,y) - \int_{[0,1[^2} \underline{f}_n(x,y) d\underline{C}^n(x,y) \right | < \epsilon.
\end{equation*}
Combining this with \eqref{bound}, we get \eqref{conv}. \qquad \qed
\end{proof}

The assumption that $f$ is continuous can, perhaps, be relaxed to the case that $f$ is $C$-continuous a.e.\ for all $C \in \mathcal{C}$. This is required to make sure that
\begin{equation*}
 \int_{[0,1[^2} f(x,y) dC(x,y)
\end{equation*}
exists for all $C \in \mathcal{C}$.\\

By defining the function families $\underline{f}_n, \overline{f}_n$ differently, we might get an approximation technique which converges faster to the optimal value, for instance we could use 
\begin{equation*}
 f_n(x,y) = f\left(\frac{i}{2^n}, \frac{j}{2^n} \right), \quad \text{ for all } (x,y) \in I^n_{i,j}.
\end{equation*}
Furthermore, the mini- and maximization steps in \eqref{mini} can be time-consuming, for instance when these problems are not explicitly solvable. However the advantage of the present approach lies in the fact that we get an upper and lower bound of the optimal value for every $n$, which is obviously useful for numerical applications.\\ 

In numerical investigations, where \eqref{mini} could not be solved explicitly, we used mini- and maximization over a fixed grid in each $I^n_{i,j}$. This results in a fast computation, however we obviously lose the property of upper and lower bounds for every $n$.\\

By assuming Lipschitz-continuity of $f$, we can describe the rate of convergence of our method.

\begin{corollary}
 Let the assumptions of Theorem \ref{main2} hold and, in addition assume that $f$ is Lipschitz-continuous on $[0,1]^2$ with parameter $L$. Then
\begin{equation*}
 \left | \int_{[0,1[^2} \overline{f}_n(x,y) d\overline{C}^n(x,y) -  \int_{[0,1[^2} \underline{f}(x,y) d\underline{C}(x,y) \right| \leq  L \frac{\sqrt{2}}{2^n}.
\end{equation*}
\end{corollary}
\begin{proof}
 Following the proof of Theorem \ref{main2} and using the Lipschitz-continuity of $f$ we get
\begin{equation*}
 | \overline{f}_n(x,y) - \underline{f}_n(x,y) | \leq L \frac{\sqrt{2}}{2^n}, \quad \text{ for all } (x,y) \in [0,1]^2,
\end{equation*}
and thus
\begin{equation*}
  \left | \int_{[0,1[^2} \overline{f}_n(x,y) d\overline{C}^n(x,y) - \int_{[0,1[^2} \underline{f}_n(x,y) d\underline{C}^n(x,y) \right | \leq  L \frac{\sqrt{2}}{2^n}.
\end{equation*}
\qed
\end{proof}

\section{Applications}

In this section we present two numerical examples in which we apply the approximation technique presented in Theorem \ref{main2}. We use an implementation of the Hungarian Algorithm in MatLab, which makes it possible to derive the solution of the linear assignment problem \eqref{lin} for a given matrix $A$ of size $2^{10} \times 2^{10}$ within seconds. The involved mini- or maximization of the integrand function on a given grid can be done efficiently, since the integrand functions are piecewise smooth.

\subsection{Uniform Distribution Theory}\label{udt}

A deterministic sequence $(x_n)_{n > 1}$ of points in $[0,1[$ is called uniformly distributed (u.d.) iff
\begin{equation*}
 \lim_{N \rightarrow \infty} \frac{1}{N} \sum_{n = 1}^N \mathbf{1}_{[a, b[} (x_n) = b-a
\end{equation*}
for all intervals $[a, b[ \subseteq [0,1[$. Furthermore, we call $g$ the asymptotic distribution function (a.d.f.) of a point sequence $(x_n,y_n)_{n > 1}$ in $[0,1[^2$ if
\begin{equation*}
 g(x,y) = \lim_{N \rightarrow \infty} \frac{1}{N}  \sum_{n = 1}^{N} \mathbf{1}_{[0,x[ \times [0,y[}(x_n,y_n),
\end{equation*}
holds in every point $(x,y)$ of continuity of $g$, for a survey of classical results in this field see \cite{sp}. In \cite{fial}, Fialov\'a and Strauch consider
\begin{equation*}
 \limsup_{N \rightarrow \infty} \frac{1}{N} \sum_{n = 1}^N f(x_n, y_n),
\end{equation*}
where $(x_n)_{n > 1}, (y_n)_{n > 1}$ are u.d.\ sequences in the unit interval and $f$ is a continuous function on $[0,1[^2$, see also \cite{ps}. In this case the a.d.f.\ $g$ of $(x_n,y_n)_{n > 1}$ is always a copula and we can write
\begin{equation}\label{inte}
 \lim_{N \rightarrow \infty} \frac{1}{N} \sum_{n = 1}^N f(x_n, y_n) = \int_0^1 \int_0^1 f(x,y) dg(x,y).
\end{equation}
Now we can derive upper bounds for \eqref{inte} by maximizing $g$ over the set of all copulas. This has already been done in \cite{fial} for functions $f$ where $\frac{\partial^2 f}{\partial_x \partial_y}(x,y)$ has constant sign for all $(x,y) \in [0,1[^2$. Note that this condition is equivalent to the two-increasing property of $f$ provided that $\frac{\partial^2 f}{\partial_x \partial_y}(x,y)$ exists on the unit square.\\ 

As a numerical example, we consider
\begin{equation*}
  \limsup_{N \rightarrow \infty} \frac{1}{N} \sum_{n = 1}^N \sin(\pi (x_n + y_n)).
\end{equation*}
The numerical results are illustrated in Table \ref{tab:t2}. Note that the approximations of the lower bound can be easily computed using the symmetry of the sine function.\\

A further interesting question concerns the sequences $(x_n)_{n > 1}, (y_n)_{n > 1}$ which maximize \eqref{inte}. Let $(x_n)_{n > 1}$ be a u.d.\ sequence and $C(x,y)$ a shuffle of $M$, then it is easy to see that $(f(x_n))_{n > 1}$ is u.d., where $f$ is the support of $C$. Thus if $C$ is the shuffle of $M$ which attains the maximum in \eqref{inte}, an optimal two-dimensional sequence is given as $(x_n, f(x_n))_{n > 1}$, where $(x_n)_{n > 1}$ is an arbitrary u.d.\ sequence. In Figure \ref{fig: max1}, we present the copula which attains the upper bound for the maximum in our approximation when $n = 7$.\\ 

Although we can not give a rigorous proof, by increasing $n$ it seems that the copula $C'$ which attains the maximum is the shuffle of $M$ with parameters $\{2, (0,0.75,1), (1), \{\omega(1) = -1, \omega(2) = 1 \}\}$. In this case we have
\begin{align*}
 \int_0^1 \int_0^1 \sin(\pi (x + y)) dC'(x,y) &= \int_0^1 \sin(\pi (x + f'(x))) dx\\
&= \int_0^{\frac{3}{4}} \sin(\pi (x + 0.75 - x)) dx + \int_{\frac{3}{4}}^1 \sin(\pi 2 x) dx\\
&= \frac{3}{4 \sqrt{2}} - \frac{1}{2 \pi} \approx 0.371175,
\end{align*}
where $f'$ denotes the support of $C'$.

\begin{table}[!ht]
\centering
\begin{tabular}[h]{|r|r|r|r|r|r|r|}
\toprule
	$n$ & 5 & 6 & 7 & 8& 9 & 10\\
UB		  & 0.3933 & 0.3824 & 0.377 & 0.3741 & 0.3727 & 0.3712\\
LB 	  & 0.3482 & 0.3598 & 0.3655 & 0.3684 & 0.3698 & 0.3711\\
\bottomrule
\end{tabular}
 \caption{Upper and lower bounds for the maximum in \eqref{inte} with respect to $n$.}
 \label{tab:t2}
\end{table}

\begin{figure}[H]
 \centering
 \includegraphics[width=7cm]{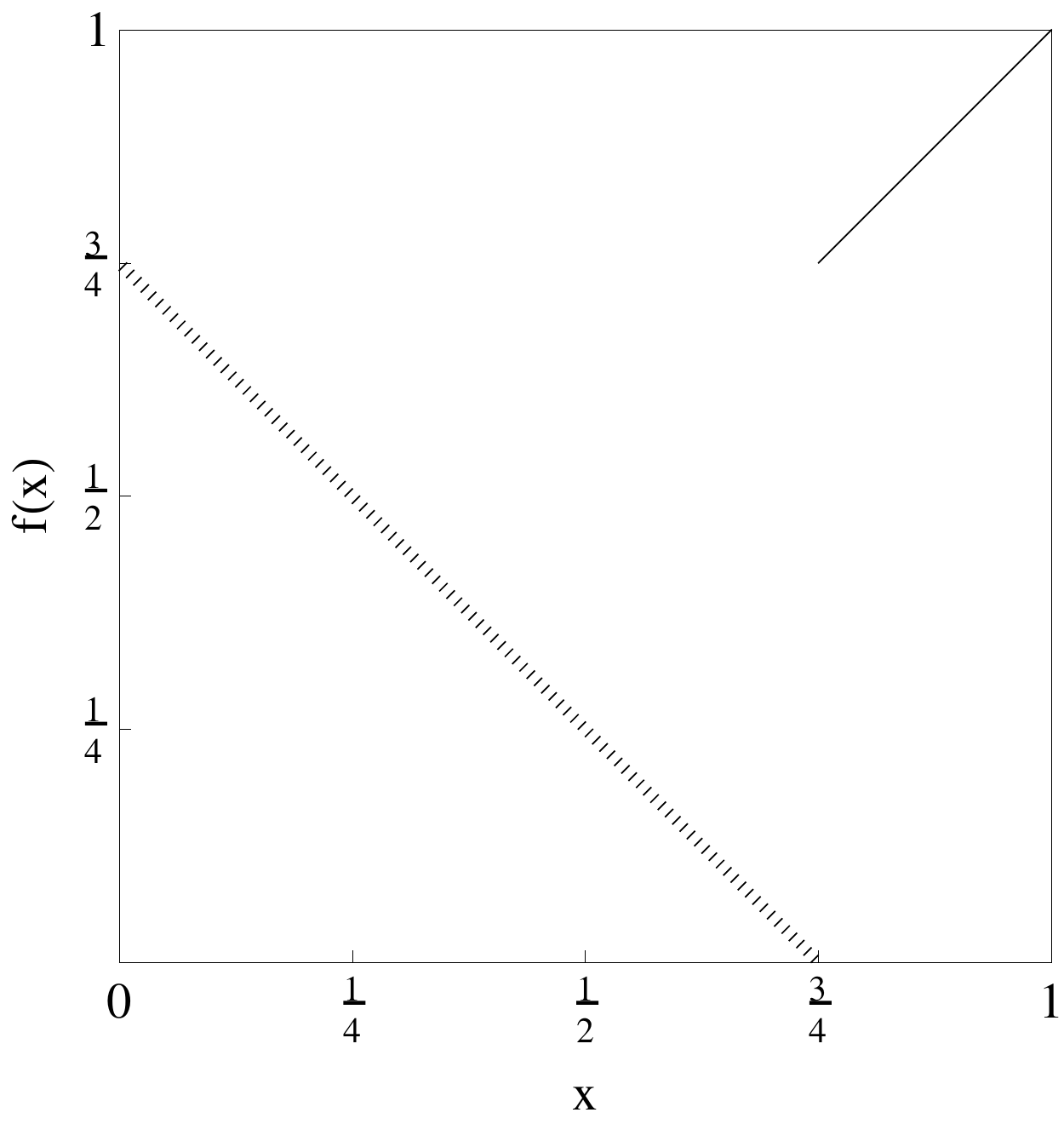}
\caption{Support of copula which attains upper bound for $\sin( \pi (X + Y))$ and $n = 7$.}
\label{fig: max1}
\end{figure}

\subsection{First-to-default Swaps}
A first-to-default swap (FTD) is a contract in which a protection seller (PS) insures a protection buyer (PB) against the loss caused by the first default event in a portfolio of risky assets. The PB pays regularly a fixed constant premium to the PS, the so-called spread, until the maturity $T$ of the contract or the first default event, whichever occurs first. In exchange, the PS compensates the loss caused by the default at the time of default.\\  

We assume that the underlying portfolio consists of two risky assets, for which the marginal default distributions are known, but the joint distribution is unknown. We want to derive a worst case bound in this setting. For the valuation of the FTD we follow the paper of Schmidt and Ward \cite{schmidt}. Note that Monte Carlo methods for the evaluation of first-to-default swaps, where the dependences within the portfolio is modeled by a copula, are e.g.\ presented in Aistleitner et al.\ \cite{aht} and Packham and Schmidt \cite{packham}.\\ 

Let $\tau_1, \tau_2$ denote the random default times of the two risky assets, let the notional be equal to one for both assets and $R_i, i = 1,2,$ be the so-called recovery rates, which are the percental amounts of money that can be liquidized in case of the default of an asset. We assume that the distribution of $\tau_i$ is given as
\begin{equation*}
 \mathbb{P}(\tau_i \leq t) = 1 - e^{-\lambda_i t}, \quad t > 0,
\end{equation*}
where the intensity $\lambda_i$ can be derived from the credit default swap market as
\begin{equation*}
 \lambda_i = \frac{s_i}{1 - R_i},
\end{equation*}
and $s_i$ is the premium of an insurance against the default of asset $i$.\\

Now denote by $\tau = \min(\tau_1, \tau_2)$ the first default time in the portfolio, let $0 = t_0 < t_1 < \ldots < t_n = T$ be the payment times of the constant spread and assume that there exists a risk free interest rate $r \geq 0$. Then, to guarantee a fair spread $s$, we obtain that the expected, discounted premium and default payments are equal, i.e.\
\begin{equation*}
 s \sum_{i = 0}^n e^{- r t_i} \mathbb{P}(\tau > t_i) = \sum_{i = 1}^2 \mathbb{E}\left[ (1- R_i) e^{- r \tau} \mathbf{1}_{\{ \tau < T \wedge \tau = \tau_i\}} \right].
\end{equation*}
By the above assumptions we obtain that
\begin{align*}
 &\mathbb{P}(\tau > t_i) = \int_{[0,1[^2} \mathbf{1}_{\left\{f(x, \lambda_1) > t_i ~\wedge ~f(y, \lambda_2) > t_i \right\}} dC(x,y),\\
 &\sum_{i = 1}^2 \mathbb{E}\left[ (1- R_i) e^{- r \tau} \mathbf{1}_{\{ \tau < T ~\wedge ~\tau = \tau_i\}} \right] = \\ 
&\int_{[0,1[^2} e^{-r \min\left( f(x, \lambda_1), f(y, \lambda_2) \right)} \biggl(\mathbf{1}_{\left\{ f(x, \lambda_1) \leq \min\left(f(y, \lambda_2), T\right) \right\}} (1 - R_1) \\ 
  &+ \mathbf{1}_{\left \{f(y, \lambda_2) \leq \min\left(f(x, \lambda_1), T\right) \right\}} (1 - R_2) \biggr) dC(x,y),
\end{align*}
where $f(x, \lambda) = \frac{- \log(1-x)}{\lambda}$ is the inverse distribution function of an exponential distribution with parameter $\lambda$ and $\mathbf{1}_{\{(x,y) \in B\}}$ denotes the characteristic function of set $B \subseteq [0,1[^2$.\\

Now we want to calculate the maximal spread $s$ by maximizing over all copulas. We obtain for the spread that
\begin{align}
 s &= \int_{[0,1[^2} \frac{e^{-r \min\left( f(x, \lambda_1), f(y, \lambda_2) \right)}}{\sum_{i = 0}^n e^{- r t_i} \mathbf{1}_{\left\{f(x, \lambda_1) > t_i ~\wedge ~f(y, \lambda_2) > t_i \right\}} }\notag \\
 &\cdot \biggl(\mathbf{1}_{\left\{ f(x, \lambda_1) \leq \min\left(f(y, \lambda_2), T\right) \right\}} (1 - R_1) \notag\\ 
 &+ \mathbf{1}_{\left \{f(y, \lambda_2) \leq \min\left(f(x, \lambda_1), T\right) \right\}} (1 - R_2)\biggr) dC(x,y). \label{ftd}
\end{align}
Note that the value of the integral is finite since the first payment is made at $t_0 = 0$. Furthermore, the integrand function in this example is not continuous, thus Theorem \ref{main2} cannot be applied. Nevertheless, it is clear that our technique provides upper and lower bounds for the optimal values, and since these bounds converge to each other our approach still works.\\

In Table \ref{tab:t1} we present numerical results for a concrete example with three payment times, $t_i = 0,1,2$. One can observe that the resulting copulas (given in Figures \ref{fig: max2} and \ref{fig: max3} for $n = 7,8$, respectively) are highly irregular in left upper quarter of the unit square. Nevertheless for $n = 10$ the upper and lower bounds for the optimal values are almost equal.

\begin{table}[!ht]
\centering
\begin{tabular}[h]{|r|r|r|r|r|r|r|r|}
\toprule
	$\lambda_1$ & $\lambda_2$ & $R_1$ & $R_2$ & $T$ & $r$ & $t_i$ &\\
$\frac{1}{3}$ & $\frac{1}{2}$	& 0.5	& 0.7	& 2 & 0.05 & (0, 1, 2)&\\	
\midrule
	$n$ & 3 & 4 & 5 & 6 & 7 & 8& 10\\
\midrule
$\overline{UB}$		& 0.3601 & 0.3355 &0.3301 &0.326 & 0.322 & 0.3202  &0.3195\\
$\overline{LB}$ 	& 0.2956  & 0.3031 & 0.314 & 0.318 & 0.3183 & 0.3189 & 0.3195\\
\midrule
$\underline{UB}$		& 0.1714  & 0.1674  & 0.1567 & 0.1535 & 0.1519  &  0.1505  & 0.1498\\
$\underline{LB}$ 	& 0.1453  & 0.1456 & 0.1458 & 0.1480  & 0.1492 & 0.1492 & 0.1495 \\
\bottomrule
\end{tabular}
 \caption{Approximation of the maximal spread of a FTD, where $\overline{UB}$ and $\overline{LB}$ and $\underline{UB}$ and $\underline{LB}$ denote the values of the upper and the lower bounds of the maximal and minimal value of the integral, and $n$ the fineness of the approximation according to Theorem \ref{main2}.}
 \label{tab:t1}
\end{table}

\begin{figure}[H]
 \centering
 \includegraphics[width=7cm]{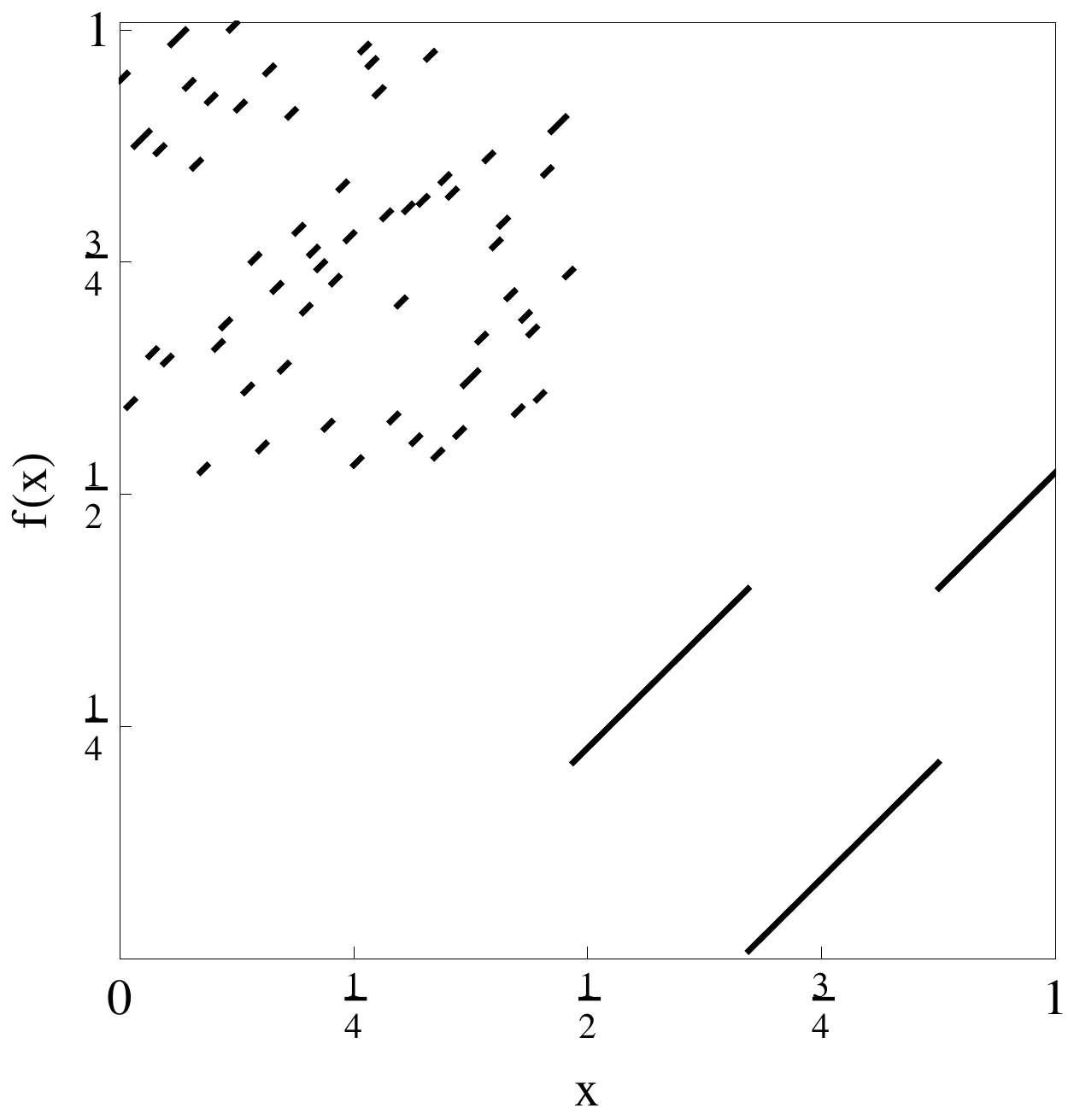}
\caption{Copula which attains the upper bound for the maximal value with $n = 7$.}
\label{fig: max2}
\end{figure}

\begin{figure}[H]
 \centering
 \includegraphics[width=7cm]{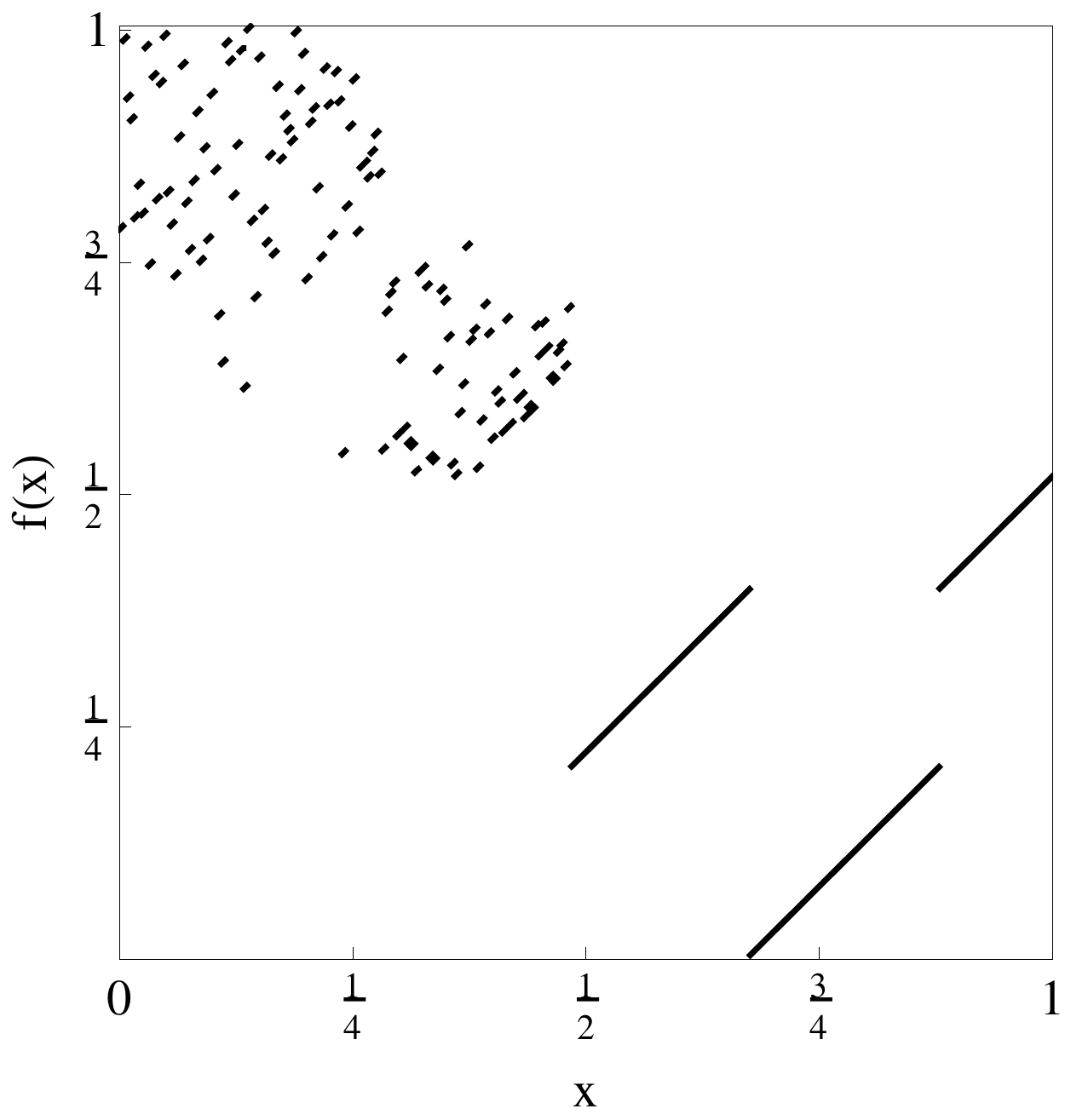}
\caption{Copula which attains the upper bound for the maximal value with $n = 8$.}
\label{fig: max3}
\end{figure}

\section{Conclusions}
The method presented in this paper can be used to derive sharp bounds for integrals of piecewise constant functions with respect to copulas. This extends the scientific literature on this topic, that is in general still open. The numerical effectiveness of our method was illustrated in two numerical examples from different branches of applied mathematics.\\

A starting point for further research is an extension of the presented technique to higher dimensional problems, since founding bounds for multidimensional integrals with respect to copulas has several applications in fields of mathematics such as number theory, financial and actuarial mathematics. Of course our aim is to study and investigate general problems and try to find a link between different branches of mathematics. Nevertheless, since the resulting so-called multi-index assignment are in general NP-hard, we plan to investigate heuristics; see e.g.\ \cite{burk}.

\section*{Acknowledgements}
The authors would like to thank Prof.\ Robert Tichy from TU Graz and Prof.\ Oto Strauch from the Slovak Academy of Science for helpful remarks and suggestions. Furthermore the authors are indebted to two anonymous referees who helped to improve the paper.


\begin{thebibliography}{}

\bibitem{tank}
Tankov, P.:
\newblock {Improved Fr\'echet bounds and model-free pricing of multi-asset options}.
\newblock {J. Appl. Probab.}, {\bf 48}, 389-403, (2011)

\bibitem{pr1}
Puccetti, G., R{\"u}schendorf, L.:
\newblock {Sharp bounds for sums of dependent risks}.
\newblock {J. Appl. Probab.}, {\bf 50}(1), 42-53, (2013)

\bibitem{packham}
Packham, N., Schmidt, W.M.:
\newblock {Latin hypercube sampling with dependence and application in
  finance}.
\newblock {J. Comput. Finance}, {\bf 13}(3), 81-111, (2010)


\bibitem{rap}
Rapuch, G., Roncalli, T.:
\newblock {Some remarks on two-asset options pricing and stochastic dependence
  of asset prices}.
\newblock {tech.\ report, Groupe de Recherche Operationelle, Credit
  Lyonnais}, (2001)

\bibitem{tchen}
Tchen, A.H.:
\newblock {Inequalities for distributions with given margins}.
\newblock {Ann. Appl. Probab.}, {\bf 8}, 814-827, (1980)

\bibitem{fial}
Fialov\'{a}, J., Strauch, O.:
\newblock {On two-dimensional sequences composed by one-dimensional uniformly
  distributed sequences}.
\newblock {Unif. Distrib. Theory}, {\bf 6}(1), 101-125, (2011)

\bibitem{aak}
Albrecher, H., Asmussen, S., Kortschak, D.:
\newblock {Tail asymptotics for dependent subexponential differences}.
\newblock {Sib. Math. J.}, {\bf 53}(6), 965-983, (2012)

\bibitem{pr3}
Puccetti, G.:
\newblock {Sharp bounds on the expected shortfall for a sum of dependent random
  variables}.
\newblock {Statist. Probab. Lett.}, {\bf83}(4), 1227-1232, (2013)

\bibitem{ru}
R{\"u}schendorf, L.:.
\newblock {Solution of a statistical optimization problem by rearrangement
  methods}.
\newblock {Metrika}, {\bf 30} 55-61, (1983)

\bibitem{nelsen}
Nelsen, R.B.:
\newblock {An Introduction to Copulas, 2nd edition.}
\newblock Springer, New York, (2006)

\bibitem{joe}
Joe, H.:
\newblock {Multivariate Models and Dependence Concepts}.
\newblock Chapman and Hall, London, (1997)

\bibitem{kuhn}
Kuhn, H.W.:
\newblock {The Hungarian method for the assignment and transportation
  problems}.
\newblock {Naval Res. Logist. Quart.}, {\bf2}, 83-97, (1955)

\bibitem{burk}
Burkard, R., Dell'Amico, M., Martello, S.:
\newblock {Assignment Problems}.
\newblock SIAM, Philadelphia, (2009)

\bibitem{Mirsky}
Mirsky, L.:
\newblock Proofs of two theorems on doubly-stochastic matrices.
\newblock {Proc. Amer. Math. Soc.}, {\bf 9}, 371-374, (1958)

\bibitem{sp}
Strauch, O., Porubsk\'{y}, $\check{S}$.:
\newblock {Distribution of Sequences: A Sampler}.
\newblock Peter Lang, Frankfurt am Main, (2005)

\bibitem{ps}
Pillichshammer, F., Steinerberger, S.:
\newblock {Average distance between consecutive points of uniformly distributed
  sequences}.
\newblock {Unif. Distrib. Theory}, {\bf 4}(1), 51-67, (2009)

\bibitem{schmidt}
Schmidt, W., Ward, I.:
\newblock {Pricing default baskets}.
\newblock {Risk}, {\bf 15}(1), 111-114, (2002).

\bibitem{aht}
Aistleitner, C., Hofer, M., Tichy, R.:
\newblock {A central limit theorem for Latin hypercube sampling with dependence and application to exotic basket option pricing}.
\newblock {Int. J. Theor. Appl. Finance},
  {\bf 15}(7), 20 pp., (2012)
\end{thebibliography}

\end{document}